%% file: main.tex
\documentclass{svproc}

\usepackage{color}
\definecolor{myblue}{rgb}{0.00000,0.44700,0.74100}%
\definecolor{myred}{rgb}{0.95000,0.200,0.200}%
\definecolor{mycolor3}{rgb}{0.92900,0.69400,0.12500}%
\definecolor{mygreen}{rgb}{0.1,0.7,0.15}%

\usepackage[hidelinks]{hyperref}
\hypersetup{
	colorlinks   = true, %Colours links instead of ugly boxes
	urlcolor     = myblue, %Colour for external hyperlinks
	linkcolor    = myblue, %Colour of internal links
	citecolor   = myblue %Colour of citations
}

\usepackage{amsmath,amssymb} 
\newtheorem{assumption}{Assumption}

\usepackage{tikz}
\usetikzlibrary{positioning}
\usetikzlibrary{calc}
\usetikzlibrary{decorations.markings,calc,patterns,decorations.pathmorphing}
\usetikzlibrary{arrows, shapes}

\newcommand{\R}{\mathbb{R}}
\newcommand{\N}{\mathbb{N}}
\newcommand{\I}{\mathbb{I}}
\newcommand{\X}{\mathbb{X}}
\newcommand{\U}{\mathbb{U}}
\newcommand{\Kinf}{\mathcal{K}_\infty}

\renewcommand{\L}{\mathcal{L}}
\newcommand{\Pistar}{{\Pi^{\star}}}
\newcommand{\PiXstar}{{\Pi^{\star}_{\X}}}

\newcommand{\betakN}[2][N]{{\textstyle \beta \left(\frac{#2}{#1}\right)}}
\newcommand{\refeq}[2]{\overset{\makebox[0pt][c]{\scriptsize #1}}{#2}}

\newcommand{\uNx}[1][x]{u_{#1}^\star} 
\newcommand{\xNx}[1][x]{x_{#1}^\star} 
\newcommand{\uNxrot}[1][x]{\tilde u_{#1}^\star}
\newcommand{\xNxrot}[1][x]{\tilde x_{#1}^\star}

\usepackage{tikz}
\usepackage{pgfplots}

\begin{document}
\mainmatter              

\title{On discount functions for economic model predictive control without terminal conditions}

\titlerunning{Discounted economic MPC}  

\author{Lukas Schwenkel\inst{1}\footnote[1]{L. Schwenkel thanks the International Max Planck Research School for Intelligent Systems (IMPRS-IS) for supporting him.}
	\and Daniel Briem\inst{1} \and Matthias A. M\"uller\inst{2} \and Frank Allg\"ower\inst{1}}

\authorrunning{L. Schwenkel, D. Briem, M. Müller, F. Allgöwer} 

\tocauthor{Lukas Schwenkel, Daniel Briem, Matthias M\"uller, Frank Allg\"ower}

\institute{University of Stuttgart, Institute for Systems Theory and Automatic Control, 70569 Stuttgart, Germany.
	\email{\{schwenkel, allgower\}@ist.uni-stuttgart.de, daniel.briem@outlook.de}
	\and Leibniz University Hannover, Institute of Automatic Control, 30167 Hannover, Germany.
	\email{mueller@irt.uni-hannover.de}
}

\maketitle

\begin{abstract}
	In this paper, we investigate discounted economic model predictive control (E-MPC) schemes without terminal conditions in scenarios where the optimal operating behavior is a periodic orbit.
	For such a setting, it is known that a linearly discounted stage cost guarantees asymptotic stability of any arbitrarily small neighborhood of the optimal orbit if the prediction horizon is sufficiently long.
	However, in some examples very long prediction horizons are needed to achieve the desired performance.
	In this work, we extend these results by providing the same qualitative stability guarantees for a large class of discount functions.
	Numerical examples illustrate the influence of the discount function and show that with suitable discounting we can achieve significantly better performance than the linearly discounted E-MPC, even for short prediction horizons.
	\keywords{Economic model predictive control, periodic optimal control, practical asymptotic stability.
	}
\end{abstract}

\input{sI_Intro_gd}

\input{sII_ProblemSetup_gd}

\input{sV_pas_gd}
\input{sVI_NumericalEx_gd}
\input{sVII_concl_gd}

\input{sVIII_appx_gd}

\end{document}

%% file: sI_Intro_gd.tex
\section{Introduction}
\label{s:intro}
Economic model predictive control (E-MPC) has gained popularity due to its ability to directly optimize a certain performance criterion while satisfying constraints for nonlinear systems \cite{11-ellis14}, \cite{25-faulwasser18}, \cite{18-gruene17}, \cite{24-koehler24}.
In E-MPC, the cost to be minimized by solving an optimal control problem could represent, e.g., the production capacity or energy demand in process control \cite{11-ellis14}.
The optimal operating behavior in such a setting is not pre-specified as in tracking MPC, but implicitly specified by the interplay of the system, the stage cost, and the constraints.
Thus, the optimal operation can result in a steady state (e.g., \cite{15-amrit11}, \cite{8-angeli11}, \cite{6-gruene13}, \cite{5-gruene14}), a periodic orbit(e.g., \cite{2-koehler18},~\cite{3-mueller16},~\cite{14-zanon16},~\cite{1-schwenkel23}), or an even more general trajectory (e.g.,~\cite{Dong2018},~\cite{22-martin19}).
To give guarantees on stability and performance, terminal conditions are used in \cite{15-amrit11} for steady state and in \cite{14-zanon16} for periodic optimal operation.
In \cite{14-zanon16}, however, the stability is only proven for one specific phase on the orbit.
Moreover, implementing terminal conditions requires significant design effort including knowledge of the optimal behavior, which is why terminal conditions are often avoided in practice. 
For E-MPC without terminal conditions in scenarios where the optimal operating behavior is a steady state, \cite{6-gruene13} and \cite{5-gruene14} provide practical asymptotic stability under usual controllability and dissipativity assumptions.
However, if the optimal operating behavior is a periodic orbit and the cost is not constant on the orbit \cite{2-koehler18}, this MPC scheme could fail to approach optimal average performance even for an arbitrarily long prediction horizon as seen in examples in \cite[Ex.~6]{1-schwenkel23} or \cite[Ex.~4, Ex.~5]{3-mueller16}.
To overcome this problem, the authors of \cite{3-mueller16} propose a $p^\star$-step E-MPC scheme, where $p^\star$ is the period length of the optimal orbit.
However, this setup requires system- and cost-specific knowledge of the orbit length and can only provide asymptotic convergence but no stability guarantees.
As a remedy, \cite{1-schwenkel23} introduced a linearly discounted stage cost function and provided guarantees on asymptotic average performance and practical asymptotic stability.
This means that any arbitrarily small neighborhood of the optimal orbit can be asymptotically stabilized if the prediction horizon length is long enough.
However, for some examples, extremely long prediction horizons may be necessary to achieve the desired performance~\cite[Ex.~19]{1-schwenkel23}.

In this work, we investigate for which discount functions we can get the same qualitative guarantees as in \cite{1-schwenkel23} and show in numerical examples that even for short prediction horizons the performance can be improved significantly.
In particular, we prove stability for a large class of discount functions, thereby enabling the control engineer to use the discount function as a tuning variable.
Further, it follows from this result that the asymptotic average performance is optimal up to an error that vanishes with growing prediction horizons.
Our numerical examples show the performance of various discounts and we observe in these examples that the linear discount from \cite{1-schwenkel23} can be significantly outperformed if, for example, only the second half of the horizon is linearly discounted.
We emphasize that this work aims to solve the undiscounted infinite-horizon optimal control problem. This differs from related works on exponentially discounted economic MPC (e.g.,~\cite{Gaitsgory2018},~\cite{Zanon22}, and references therein), which focus on solving the exponentially discounted infinite-horizon optimal control problem.

\emph{Outline.} Section~\ref{s:pr-set} introduces the problem setup and the discounted E-MPC scheme. Section~\ref{s:pas} contains the main results, while the technical parts of the proof are moved to the Appendix~\ref{app:lemmas}.
The numerical examples are presented in Section~\ref{s:NumEx} and our conclusion are drawn in Section~\ref{s:cao}.

\emph{Notation.}
In this work, $\mathbb{N}$, $\R$, and $\mathbb{R}_0^+$ denote the sets of naturals, reals, and non-negative reals.
Integers in the interval $[a,\,b]$ are denoted by $\mathbb{I}_{[a,b]}$ for $a\leq b$.
The modulo operator $[k]_p$ returns the remainder when dividing $k\in\mathbb{N}$ by the natural $p\geq 1$.
The floor-operator $\lfloor x \rfloor$ returns the largest integer that is smaller than or equal to $x\in\mathbb{R}$ and respectively $\lceil x \rceil$ as the ceiling-operator denotes the smallest integer that is larger than or equal to $x$.
The function class $\mathcal{K}$ contains functions $\alpha : \mathbb{R}_0^+ \to \mathbb{R}_0^+$, which are continuous and strictly increasing with $\alpha(0) = 0$. It is $\alpha \in \mathcal{K}_\infty$ if $\alpha$ is additionally unbounded. We denote $\delta \in \mathcal{L}$ if $\delta : \mathbb{N} \to \mathbb{R}_0^+$ is non-increasing and $\lim_{k\to\infty} \delta(k) = 0$. 
Further, $\gamma \in \mathcal{KL}$ denotes functions $\gamma : \mathbb{R}_0^+ \times \mathbb{N} \to \mathbb{R}_0^+$ with the property $\gamma(\cdot,\,t) \in \mathcal{K}$ and $\gamma(r,\,\cdot) \in \mathcal{L}$.

%% file: sII_ProblemSetup_gd.tex
\section{Problem setup}
\label{s:pr-set}
We consider a nonlinear discrete-time system
\begin{equation}
x(k+1) = f\big(x(k),u(k)\big) , \qquad x(0) = x_0
\label{eq:sys}
\end{equation}
with state and input constraints $x(k) \in \mathbb{X} \subseteq \mathbb{R}^n$ and $u(k) \in \mathbb{U} \subseteq \mathbb{R}^m$.
The solution of \eqref{eq:sys} at time step $k$ resulting from an input sequence $u \in \mathbb{U}^K$ of length $K\in\mathbb{N}$ and initial condition $x_0 \in \mathbb{X}$ is denoted by $x_u(k,x_0)$.
The input sequence $u\in\U^K$ is feasibly, denoted by $u\in\U^K(x_0)$, if $x_u(k,x_0) \in \mathbb{X}$ for all $k\in\I_{[0,K]}$.
We aim to operate the system optimally concerning a (not necessarily positive definite) stage cost $\ell: \mathbb{X} \times \mathbb{U} \to \mathbb{R}$, in particular, we want to minimize the asymptotic average performance
\begin{equation}
J_\infty^\mathrm{av} (x,u) := \limsup_{T \to \infty} \frac{1}{T} J_T(x,u)\ \ \text{ with } \ \ 
J_T(x,u) := \sum_{k=0}^{T-1} \ell\big(x_u(k,x), u(k)\big)
\label{eq:J_av}
\end{equation}
and to stabilize the optimal operating behavior. We assume that the optimal operating behavior is a periodic orbit, which is reachable and controllable.
\begin{definition}[Periodic orbit]
	Let $\Pi_\X:\I_{[0,p-1]}\to \X$ and $\Pi_\U:\I_{[0,p-1]}\to \U$ for $p\in\N$. Then $\Pi= (\Pi_\X,\Pi_\U)$ is called a \emph{$p$-periodic orbit} of system~\eqref{eq:sys}, if 
	\begin{align}
		\Pi_\X \big([k+1]_p\big) = f\big(\Pi_\X (k), \Pi_\U(k) \big) = f\big(\Pi(k)\big)
	\end{align}
	for all $k \in \I_{[0,p-1]}$. 
	A $p$-periodic orbit $\Pi$ is called minimal, if $\Pi_\X$ is injective.
	The distance of a point $(x,u)\in\X\times \U$ to the image of the function $\Pi$ is defined by $\|(x,u)\|_\Pi := \min_{k\in \I_{[0,p-1]}} \|(x,u)-\Pi(k)\|$.
\end{definition}

\begin{assumption}[Optimal periodic orbit]\label{ass:orbit}
	There exists a minimal $p^\star$-periodic orbit $\Pistar$ that satisfies 
	for any $p\in\N$ and any $p$-periodic orbit $\Pi$ that
	\begin{align}\label{eq:ellstar}
		\ell^\star:=\frac 1 {p^\star} \sum_{k=0}^{{p^\star}-1} \ell (\Pi^\star (k)) \leq \frac 1 p \sum_{k=0}^{p-1} \ell (\Pi(k)).
	\end{align}
\end{assumption}
The following assumptions characterize the problem class for which we can provide performance and stability guarantees. These assumptions are standard in the context of periodic E-MPC (cf.~\cite{2-koehler18}, \cite{3-mueller16},~\cite{14-zanon16},~\cite{1-schwenkel23}).
\begin{assumption}[Strict dissipativity]\label{ass:s-diss}
There exist $\underline{\alpha}_{\tilde{\ell}} \in \mathcal{K}_\infty$ and a continuous storage function $\lambda : \mathbb{X} \to \mathbb{R}$ such that for all $x \in \mathbb{X}$ and $u \in \mathbb{U}^1(x)$ we have
\begin{equation}
\tilde{\ell}(x,u) := \ell(x,u) - \ell^\star + \lambda(x) - \lambda\big(f(x,u)\big) \geq \underline{\alpha}_{\tilde{\ell}}(\lVert(x,u)\rVert_{\Pi^\star}).
\label{eq:s-diss}
\end{equation}
We call $\ell(x,u) - \ell^\star$ the supply rate and $\tilde \ell$ the rotated stage cost.
\end{assumption}
\begin{assumption}[Continuity and compactness]
The functions $f$ and $\ell$ are continuous, and the constraint set $\mathbb{X}\times\mathbb{U}$ is compact. Hence, $\bar{\lambda}:= \sup_{x\in\X} |\lambda(x)|$ and $\bar \ell:= \sup_{(x,u)\in \mathbb{X}\times\mathbb{U}} |\ell(x,u)|$  are finite. 
\label{ass:cac}
\end{assumption}

\begin{assumption}[Local controllability at $\PiXstar$]
	\label{ass:lctrb}
There exists $\kappa > 0, \ M_1 \in \mathbb{N}, \ \alpha_c \in \mathcal{K}_\infty$ such that for all $z \in \Pi_\mathbb{X}^\star$ and $x,\, y \in \mathbb{X}$ with $\lVert x-z \rVert \leq \kappa$ and $\lVert y-z \rVert \leq \kappa$ a control input sequence $u \in \mathbb{U}^{M_1}(x)$ exists with $x_u(M_1,x) = y$ and
\begin{equation}
\big\lVert\big(x_u(k,x),u(k)\big)\big\rVert_{\Pi^\star} \leq \alpha_c\big(\mathrm{max} \big\{\lVert x \rVert_{\Pi_\mathbb{X}^\star}, \lVert y \rVert_{\Pi_\mathbb{X}^\star}\big\}\big), \quad \forall k \in \mathbb{I}_{[0,\,M_1-1]}.
\label{eq:A-lC}
\end{equation}
\end{assumption}

\begin{assumption}[Finite-time reachability of $\Pi^\star$]
For $\kappa > 0$ from Assumption~\ref{ass:lctrb}, $\exists M_2 \in \mathbb{I}_{\geq 1}$ such that $\forall x \in \mathbb{X}$ there exist $M\in \mathbb{I}_{[0,\,M_2]}$ and an input sequence $u \in \mathbb{U}^M(x)$ such that $\lVert x_u(M,\,x) \rVert_{\Pi_{\mathbb{X}}^\star} \leq \kappa$.
\label{ass:ftr}
\end{assumption}
\begin{remark}\label{rem:ass}
	We assume finite-time reachability of $\Pistar$ for all $x\in\X$ for simplicity of our proofs, similar to~\cite{3-mueller16}. 
	As detailed in~\cite[Section~4]{1-schwenkel23} for the linearly discounted E-MPC scheme, this assumption can be relaxed to hold only on a subset of $\X$.
	A similar relaxation is also possible in the setting with general discounts, which is, however, beyond the scope of the present paper.
\end{remark}

\subsection{Discounted Economic MPC}
\label{ss:dmpcs}
In this section, we describe the discounted economic MPC scheme without terminal conditions.
Given the current measurement $x=x(t)$, we solve at every time-step $t$ the optimal control problem
\begin{align}
	V_N^\beta(x) =\min_{u\in\U^N(x)}
	J_N^\beta(x,u) = \min_{u\in\U^N(x)} \sum_{k=0}^{N-1} \betakN{k} \ell\big(x_u(k,x),u(k)\big).
	\label{eq:VNbeta}
\end{align} 
Assumption~\ref{ass:cac} guarantees that for all $x\in\X$ and all $N\in\N$ a (possibly non-unique) optimal control sequence $\uNx$ exists, i.e, $V_N^\beta(x)=J_N^\beta(x,\uNx)$.
For convenience of notation, we will also use $\xNx(k):= x_{\uNx}(k,x)$.
Although this lucid notation hides it, both $\uNx$ and $\xNx$ depend on the prediction horizon length $N$ and the chosen discount function $\beta$.
The MPC feedback we apply is $\mu_{N}^\beta(x)=\uNx (0)$.
The discount function $\beta:[0,1]\to[0,1]$ defines a uniform shape of the discount for all horizon lengths $N$.
This is crucial as the following example shows.

\begin{example}[A non-uniform discount]
\label{ex:nclc}
Let us consider the example $x(t+1)=u(t)$ with only three feasible states $\X=\{-1,0,1\}$ and four feasible transitions $\{(-1,-1), (-1,0), (0,1), (1,0)\}\subseteq \X\times \U $.
Thus, the only possible decision to make is at $x=-1$ whether to go to $x=0$ and alternate at the $2$-periodic orbit between $x=0$ and $x=1$ or whether to stay at $x=-1$. 
The cost values are $\ell(-1,-1)=1$, $\ell(-1,0)=1$, $\ell(0,1)=0$, $\ell(1,0)=1.5$ and hence the system is optimally operated at the periodic orbit with average cost $\ell^\star = 0.75$, which is better than staying at $x=-1$ for cost $1$.
However, when considering optimal trajectories of odd length $N$, it is better to wait one step at $x=-1$ for cost $1$ and then approach the orbit rather than approaching it directly and ending with the step of cost $1.5$ on the orbit. 
It was discovered in~\cite[Example~4]{3-mueller16} that this phenomenon causes an undiscounted E-MPC scheme with odd prediction horizon length to wait at $x=-1$ forever.
In~\cite[Example~6]{1-schwenkel23} it was shown that linear discounts $\beta^\text{lin} (\frac{k}{N})=1-\frac{k}{N}$ solve this problem and that a linearly discounted E-MPC goes directly to the optimal orbit for all prediction horizons $N\geq 2$.
Let us consider a slight modification of this linear discount $\beta_N(k-1)=\beta_N(k)=\beta^\text{lin}(\frac{k}{N})$ for all $k\in 2\N$.
We introduced the notation $\beta_N(k)$ as this discount is not uniform, i.e., there exists no function $\beta:[0,1]\to[0,1]$ such that $\beta_N(k) = \betakN{k}$ for all $N\in\N$ and $k\in\I_{[0,N-1]}$.
Due to the special shape of the discount, when comparing the strategy of waiting one step with the one of directly going to the orbit, we need to compare only the cost for $k\in\{0,1,2\}$ as for $k\geq 3$ both strategies are on the optimal orbit and contribute the same cost as $\beta_N(k-1)=\beta_N(k)$ for all $k\in 2\N$ and the horizon length is effectively always odd as for even $N$ we have $\beta_N(N-1)=0$.
Waiting one step has a cost of $1 + \frac{N-2}{N}$ in these first $3$ steps and is hence the optimal strategy for all $N\in \N$ as going directly to the orbit costs $1+1.5\frac{N-2}{N}$.
As MPC applies only the first step of the optimal strategy it ends up waiting forever. 
When using a uniform discount, this problem cannot occur. 

\end{example}

The following assumption confines the choice of the discount function $\beta$.
This is required to guarantee practical asymptotic stability of the discounted E-MPC.

\begin{assumption}[Discount function]\label{ass:df}
	The discount function $\beta:[0,1]\to[0,1]$ is non-increasing, piecewise continuously differentiable, Lipschitz-continuous with Lipschitz constant $L_\beta$, and satisfies $\beta(0)=1$, $\beta(1)=0$, and $\mathrm{lim\,sup}_{\xi\to 1} \beta'(\xi)<0$.
\end{assumption}

The technical assumption $\mathrm{lim\,sup}_{\xi\to 1} \beta'(\xi)<0$ is made to simplify our proofs and does not pose a relevant constraint in practice.

\begin{example}[Discount functions]
	Discount functions that satisfy Assumption~\ref{ass:df} are, for example, the linear discount $\beta^\text{lin}(\xi) = 1-\xi$ from~\cite{1-schwenkel23}, the half-undiscounted-half-linear discount $\beta^\text{half-lin}(\xi) = \min\{1,2-2\xi\}$, the polynomial discount $\beta^\mathrm{poly}(\xi) = 1-\xi^q$ with degree $q\in\N$, and many more.
	The undiscounted discount $\beta^\text{un}(\xi)=1$ for all $\xi\in[0,1]$	does not satisfy $\beta(1)=0$.
	\label{ex:dfs}
\end{example}

%% file: sV_pas_gd.tex
\section{Main result}
\label{s:pas}
The main result of this work is that a discounted E-MPC scheme without terminal conditions asymptotically stabilizes a set that contracts to the optimal orbit $\Pi^\star$ with increasing prediction horizons if the discount satisfies Assumption~\ref{ass:df}. 

\begin{theorem}[Practical asymptotic stability]
	\label{thm:PAS}
Let Assumption~\ref{ass:orbit}-\ref{ass:df} hold.
Then, there exists $\varepsilon \in \mathcal{L}$ such that $\Pi^\star$ is practically asymptotically stable w.r.t. $\varepsilon(N)$ under the discounted E-MPC feedback $\mu_N^\beta$ for all $N\in\N$, i.e., there exists $\gamma \in \mathcal{KL}$ such that for all $ x\in \mathbb{X}$ and $k\in\mathbb{N}$ we have
\begin{align*}
	\left\lVert\left(x_{\mu_{N}^\beta} (k,x),\mu_N^\beta \big(x_{\mu_N^\beta}(k,x)\big)\right)\right\rVert_{\Pi^\star} \leq \max\big\{\gamma(\lVert x \rVert_{\Pi_\mathbb{X}^\star},k),\varepsilon(N)\big\}.
\end{align*}
\end{theorem}
\begin{proof}
	Throughout the proof, we assume without loss of generality that $\ell$ is non-negative (cf.~\cite[Remark~20]{1-schwenkel23}). 
	Analogous to $J_N^\beta$, $V_N^\beta$, $\uNx$, and $\xNx$, we define $\tilde J_N^\beta$, $\tilde V_N^\beta$, $\uNxrot$, and $\xNxrot$ for the case by replacing $\ell$ by the rotated stage cost $\tilde \ell$ from~\eqref{eq:s-diss}.
	We prove the statement by showing that the rotated value function $\tilde V_N^\beta$ is a practical Lyapunov function in the sense of~\cite[Definition~13]{1-schwenkel23}.
	Showing positive definiteness of $\tilde \ell$ and $\tilde V_N^\beta$ is not different from the linearly discounted case~\cite[Lemma~27, Lemma~28]{1-schwenkel23} and is therefore not shown again here.
	The challenging part is to show that there exists $\delta\in\L$ such that the Lyapunov decrease condition 
	\begin{align}\label{eq:prac_lyap_decr2}
		\tilde V_N^\beta\left(x_{\mu_N^\beta} (1,x)\right) - \tilde V_N^\beta (x)  \leq -\tilde \ell \big(x, \mu_{N}^\beta(x)\big)  + \delta(N)
	\end{align}
	holds for all $x\in\X$. 
	To establish this inequality, we need the following two Lemmas, which are proven in the Appendix~\ref{app:lemmas}.	
	
	\begin{lemma}[Almost optimal candidate solution]\label{lem:almost_optimal_candidate}
		Let Assumptions~\ref{ass:orbit}--\ref{ass:df} hold. Then there exists $\delta_2\in\L$ such that for all $x\in\X$ and $N\in \N$ we have
		\begin{align}\label{eq:almost_optimal_candidate}
			V_N^\beta\big(\xNx(1)\big) - V_N^\beta (x)\leq -\ell\big(x,\uNx(0)\big) + \ell^\star + \delta_2(N).
		\end{align}
	\end{lemma}
	\begin{lemma}[Difference of value functions]\label{lem:diff_value}
		\label{thm:Lemma29}
		Let Assumptions~\ref{ass:orbit}-\ref{ass:df} hold. Then, there exists $\delta_3\in\mathcal{L}$ such that for all $x,y\in\mathbb{X}$, $N\in\N$ the following inequality holds
		\begin{equation}
			\tilde{V}_N^\beta(y)  - \tilde{V}_N^\beta(x) \leq \lambda(y) - \lambda(x)+ V_N^\beta(y) - V_N^\beta(x) + \delta_3(N).
			\label{eq:Lemma29}
		\end{equation}
	\end{lemma}
	These two Lemmas yield
	\begin{align*}
		\tilde V_N^\beta\big(\xNx (1)\big) - \tilde V_N^\beta (x)  &\refeq{\eqref{eq:Lemma29}}{\leq} \lambda\big(\xNx(1)\big) - \lambda(x) + V_N^\beta\big(\xNx(1)\big) - V_N^\beta(x) + \delta_3(N)\\
		&\refeq{\eqref{eq:almost_optimal_candidate}}{\leq}\lambda\big(\xNx(1)\big) - \lambda(x) -\ell\big(x,u_x^\star (0)\big)+\ell^\star + \delta_2(N)+\delta_3(N)\\
		&= -\tilde \ell\big(x,\uNx(0)\big)+\delta(N)
	\end{align*}
	with $\delta:=\delta_2+\delta_3\in\L$.
	Since $\uNx(0)=\mu_N^\beta(x)$ and  $\xNx(1)=x_{\mu_N^\beta} (1,x)$, this establishes~\eqref{eq:prac_lyap_decr2} and hence~\cite[Theorem~14]{1-schwenkel23} guarantees practical asymptotic stability with respect to $\varepsilon(N)$ for some $\varepsilon\in\L$.  \hfill \qed 
\end{proof}
Due to Theorem~\ref{thm:PAS} we directly obtain that the asymptotic average performance is optimal up to an error vanishing with growing prediction horizons.
\begin{corollary}[Asymptotic average performance]\label{cor:aap}
	Let Assumption~\ref{ass:orbit}--\ref{ass:df} hold. Then there exists $\varepsilon_2\in\L$ such that the MPC feedback $\mu_N^\beta$ achieves an asymptotic average performance of $J_\infty^\mathrm{av} (x,\mu_N^\beta)\leq \ell^\star + \varepsilon_2(N)$.
\end{corollary}
\begin{proof}
	For given $N$ and $x$, let $T_1\in\N$ be large enough such that $\gamma(\|x\|_\Pistar, T_1)\leq \varepsilon(N)$. 
	For all $T\geq T_1$ we split $J_T(x,\mu_N^\beta)$ in the interval $\I_{[0,T_1-1]}$ and the interval $\I_{[T_1,T-1]}$ in which $\big(\xNx(k),\uNx(k)\big)$ is inside the $\varepsilon(N)$ neighborhood of $\Pistar$ due to Theorem~\ref{thm:PAS}. 
	We use $\bar \ell$ to bound the first interval as well as $\varepsilon_2\in\mathcal L$ and the bound from \cite[Lemma~24]{1-schwenkel23} to bound the second interval, i.e., 
	\begin{align*}
		\frac 1 T J_T(x,\mu_N^\beta) \leq \frac {T_1}T \bar\ell + \frac{T-T_1+p^\star-1}T \ell^\star + \varepsilon_2(N).
	\end{align*}
	Now, apply $\limsup_{T \to \infty}$ to both sides to obtain the desired result.  \hfill \qed
\end{proof}
\begin{remark}
	In~\cite{1-schwenkel23}, the same asymptotic average performance guarantee was given without assuming minimality of $\Pi^\star$ and continuity of $\lambda$.
	This required an individual proof, as it could not be stated as corollary of the practical asymptotic stability result, which assumed minimality of $\Pi^\star$ and continuity of $\lambda$.
	To keep this paper concise, we did not present such a proof but we conjecture that the performance bound could similarly be established without these assumptions.
\end{remark}

%% file: sVI_NumericalEx_gd.tex
\section{Numerical Example}
\label{s:NumEx}

In this section, we present two numerical examples, compare different discount functions, and compare different E-MPC schemes for periodic operation.

\begin{example}[Periodic orbit]
\label{ex:peror}
Consider the scalar system $x(t+1) = -x(t)+u(t)$ equipped with the stage cost $\ell(x,u) = x^3$ and the constraint sets $\mathbb{X}=[-1,1]$ and $\mathbb{U} = [-0.1,0.1]$. 
The system is optimally operated at the $2$-periodic orbit $\Pistar(0) = (-1,-0.1)$, $\Pistar(1)=(0.9,-0.1)$ with the average cost $\ell^\star = \frac 1 2 (0.9^3-1)\approx -0.136$.
In Fig.~\ref{f:aap}, we compare the asymptotic average performance of the proposed discounted E-MPC scheme, starting from the initial condition $x_0=0.05$ for the different discount functions $\beta \in\{\beta^\mathrm{lin},\beta^\mathrm{half-lin},\beta^\mathrm{un},\beta^\mathrm{poly}\}$ from Example~\ref{ex:dfs}, where we chose degree $q=2$ for the polynomial discount.
We observe that the undiscounted scheme does not converge to the optimal average performance for all prediction horizon lengths $N\neq 3$. 
For all odd prediction horizons $N\geq 9$ the scheme gets stuck in the steady state $(0.05,0.1)$ and for all even prediction horizons $N\geq 10$ the scheme gets stuck in the periodic orbit $\Pi_s(0)=(0.05, -0.1)$, $\Pi_s(1)=(-0.15, -0.1)$.
This failure is not surprising, as $\beta^\text{un}$ does not satisfy Assumption~\ref{ass:df} and hence, our theoretical guarantees do not hold for $\beta^\text{un}$.
Among the other discounts, which satisfy Assumption~\ref{ass:df}, we can see that $\beta^\text{half-lin}$ performs best and achieves optimal performance for all $N\geq 7$, with $\beta^\text{poly}$ we need $N\geq 9$, and with $\beta^{\text{lin}}$ we need $N\geq 11$.

\begin{figure}[t]
	\centering
	\input{aap.tex}
	\caption{Performance comparison for different discounts in Example~\ref{ex:peror}.}
	\label{f:aap}
\end{figure}
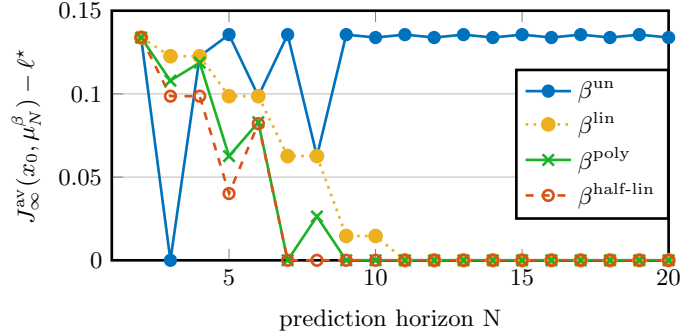

For comparison, we take a look at the $p^\star$-step E-MPC scheme from~\cite{3-mueller16}, which converges to the optimal orbit for all prediction horizon lengths $N\geq 3$ in this example.
Hence, in this aspect it is better than the three discounted E-MPC schemes, however, it needs to know the optimal period length $p^\star$ in advance, provides only convergence but no stability guarantees, and as discussed in~\cite[Example~18]{1-schwenkel23} may lead to suboptimal transient performance.

Moreover, we compare to an E-MPC with terminal equality constraints (a special case of~\cite{14-zanon16}), i.e., at time $t$ we solve~\eqref{eq:VNbeta} with $\beta=\beta^\text{un}$ and the additional constraint that $x_u(N,x)=\PiXstar([\phi_\text{end}+t]_2)$ with phase $\phi_\text{end}\in\{0,1\}$.
We denote the resulting time-dependent MPC feedback by $\mu_N^\mathrm{TC}(x, t)$.
If we start at $t=0$ at $x_0 = \PiXstar(\phi_0)$ with phase $\phi_0$, then either $[\phi_0+N]_2=\phi_\text{end}$ or $[\phi_0+N]_2\neq\phi_\text{end}$.
In the first case, the OCP is feasible for all horizon lengths $N\geq 1$ and the closed-loop follows the optimal orbit $\Pistar$.
In the second case, however, we need $N\geq 20$ for feasibility in order to leave $\Pi^\star$ and approach it with the synchronized phase again, which additionally leads to a rather high transient cost.
For example, starting at $\phi_0=0$, i.e., $x_0=-1$, choosing $N=20$ and $\phi_\text{end}=1$, leads to $J_{20}\left(x_0,\mu_{20}^\mathrm{TC}\right)=-1$ and $J_{21}\left(x,\mu_{20}^\mathrm{TC}\right)=-0.271$ in closed loop, whereas $J_{20}\left(x_0,\mu_7^{\beta^{\mathrm{half-lin}}}\right)=-2.71$ and $J_{21}\left(x,\mu_7^{\beta^{\mathrm{half-lin}}}\right)=-3.71$. 

\end{example}

\begin{example}[Economic growth]
	\label{ex:ecgr}
	Consider the scalar system $x(t+1) = u(t)$, which is equipped with the stage cost $\ell(x,\,u) = -\log(5x^{0.34}-u)$, constrained to the sets $\mathbb{X} = [0,10]$, $\mathbb{U} = [0.1,10]$, and initialized at $x(0)=0.1$.
	It is known that the optimal cost $\ell^\star \approx -1.47$ is achieved at the steady state $\Pi^\star \approx (2.23,2.23)$, i.e., we have an optimal orbit of length $p^\star=1$ (compare~\cite[Example~5.1]{5-gruene14}).
	Therefore, although $\beta^\text{un}$ does not satisfy Assumption~\ref{ass:df}, we still have guaranteed practical stability from~\cite{5-gruene14}.
	In Fig.~\ref{f:comp132}, we compare the asymptotic average performance for the discount functions $\beta^\mathrm{un}$, $\beta^\mathrm{lin}$, $\beta^\mathrm{half-lin}$,  and $\beta^\mathrm{poly}$ with degree $q=2$.
	As known from~\cite{5-gruene14} and~\cite[Example~19]{1-schwenkel23}, we can see that the performance converges exponentially in $N$ for $\beta^\mathrm{un}$ and with the approximate rate $\sim \frac {1}{N^2}$ for $\beta^\mathrm{lin}$.
	Further, we can see that $\beta^\mathrm{half-lin}$ is able to restore the exponential convergence rate and thus, significantly outperforms $\beta^\mathrm{lin}$ and $\beta^\mathrm{poly}$.
	For example, for an error of less than $10^{-10}$ we need $N\geq 11$ with $\beta^\text{un}$, $N\geq 17$ with $\beta^\text{half-lin}$ and $N\geq 6\cdot 10^4$ for $\beta^\text{lin}$ (compare~\cite{1-schwenkel23}).
	
	\begin{figure}[t!]
		\centering
		\input{comp_132.tex}
		\caption{Performance comparison for different discounts in Example~\ref{ex:ecgr}.}
		\label{f:comp132}
	\end{figure}
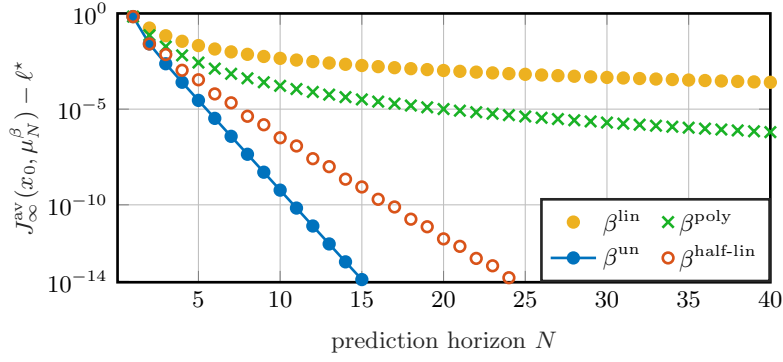
	
\end{example}

%% file: aap.tex
\definecolor{mycolor1}{rgb}{0.00000,0.44700,0.74100}%
\definecolor{mycolor2}{rgb}{0.85000,0.32500,0.09800}%
\definecolor{mycolor3}{rgb}{0.92900,0.69400,0.12500}%
%\definecolor{mycolor4}{rgb}{0.49400,0.18400,0.55600}%
\definecolor{mycolor4}{rgb}{0.1,0.7,0.15}%
\definecolor{grey}{rgb}{0.6100,0.6100,0.6100}%
\begin{tikzpicture}
	
	\begin{axis}[%
		width=2.9in,
		height=1.3in,
		ylabel near ticks,
		at={(0.758in,0.481in)},
		scale only axis,
		line width=1,
		xmin=1,
		xmax=20,
		xlabel={prediction horizon N},
		ytick={0,0.05,0.1,0.15},
		yticklabels={0,0.05,0.1,0.15},
		ymin=0,
		ymax=0.15,
		ymajorgrids,
		ylabel={$J_\infty^\mathrm{av}(x_0,\mu_N^\beta )-\ell^\star$},
		axis background/.style={fill=white},
		legend style={legend cell align=left, align=left, yshift=-0.7cm,xshift=0.16cm}
		]
\addplot [color=mycolor1,mark=*,solid,line width=1]
  table[row sep=crcr]{%
2	0.133875014620245\\
3	3.3070324256812e-11\\
4	0.122634495421876\\
5	0.135625015023722\\
6	0.0986268641346905\\
7	0.135625015023801\\
8	0.0626246051404688\\
9	0.135625015023828\\
10	0.133875014648966\\
11	0.135625015023841\\
12	0.133875014649067\\
13	0.135625015023841\\
14	0.133875014649067\\
15	0.135625015023841\\
16	0.133875014649067\\
17	0.135625015023841\\
18	0.133875014649067\\
19	0.135625015023841\\
20	0.133875014649067\\
21	0.135625015023841\\
22	0.133875014649067\\
23	0.135625015023841\\
24	0.133875014649067\\
25	0.135625015023841\\
26	0.133875014649067\\
27	0.135625015023841\\
28	0.133875014649067\\
29	0.135625015023841\\
30	0.133875014649067\\
};
\addlegendentry{$\beta^\text{un}$}

\addplot [color=mycolor3,dotted,mark=*,line width=1, mark options={solid, mycolor3, line width=1.5pt}] % lin
table[row sep=crcr]{%
	2	0.133875014591503\\
	3	0.122632771046457\\
	4	0.122633453287934\\
	5	0.0986253075333653\\
	6	0.0986257857097516\\
	7	0.0626232938517121\\
	8	0.0626235159150876\\
	9	0.0146217820637248\\
	10	0.0146211776122067\\
	11	2.23895346707081e-11\\
	12	3.10024228511452e-11\\
	13	1.00878194686516e-11\\
	14	1.14361853320588e-11\\
	15	9.82969261542621e-12\\
	16	1.17302001445552e-11\\
	17	1.0189460386556e-11\\
	18	1.1139533739879e-11\\
	19	4.67126337611035e-12\\
	20	6.66844357510854e-12\\
	21	-6.63935573186336e-12\\
	22	5.47129008765523e-12\\
	23	4.32326396904159e-12\\
	24	5.69722047316645e-12\\
	25	8.35076452432304e-12\\
	26	1.02363673093464e-11\\
	27	-5.73008307469536e-12\\
	28	2.67005861864789e-12\\
	29	-5.7021054544748e-13\\
	30	-6.82567891097108e-12\\
};
\addlegendentry{$\beta^\text{lin}$}

\addplot [color=mycolor4, solid, line width=1pt, mark size=3pt, mark=x, mark options={solid, mycolor4, line width=1pt}]
table[row sep=crcr]{%
	2	0.133875014611413\\
	3	0.107959694806669\\
	4	0.118791235766205\\
	5	0.0626229697451881\\
	6	0.0829785156411119\\
	7	1.91527349535647e-10\\
	8	0.0263265363477384\\
	9	9.64006652282023e-12\\
	10	3.14107073684511e-12\\
	11	5.6772364587232e-12\\
	12	4.11393141774852e-12\\
	13	-6.00519634019747e-12\\
	14	-5.96478422210112e-12\\
	15	-3.53700402300205e-12\\
	16	-7.96207544340177e-12\\
	17	-6.79556411142812e-12\\
	18	-4.26336743686306e-12\\
	19	-7.92407806038398e-12\\
	20	-3.9446501620688e-12\\
	21	-4.12422873630192e-12\\
	22	1.68905445185885e-11\\
	23	6.41514619204031e-12\\
	24	-3.01225711041297e-12\\
	25	8.17224066196331e-12\\
	26	-6.25266505238642e-12\\
	27	-3.714140106581e-12\\
	28	-3.67839092518807e-12\\
	29	-6.29918339711821e-12\\
	30	4.90921192586313e-12\\
};
\addlegendentry{$\beta^\text{poly}$}
\addplot [color=mycolor2,mark=o,dashed,line width=1pt, mark options={solid, mycolor2, line width=1pt}]
  table[row sep=crcr]{%
2	0.133875014620427\\
3	0.0986247836505859\\
4	0.0986258227126687\\
5	0.0401224682110728\\
6	0.0821250562026801\\
7	7.86826159782095e-12\\
8	1.48809853328657e-11\\
9	3.66706665033689e-13\\
10	1.40381040125703e-11\\
11	6.12399020383236e-13\\
12	5.95945515158292e-12\\
13	-3.61516372393567e-12\\
14	-7.39908134761436e-12\\
15	7.23776594213632e-12\\
16	-3.86116139061699e-12\\
17	-3.28004290395256e-12\\
18	-3.55870888313348e-12\\
19	-6.95093982372441e-12\\
20	-3.73301389799963e-12\\
21	-6.89112655827273e-12\\
22	-3.91178756053989e-12\\
23	-2.93765012315816e-12\\
24	4.72355488057019e-12\\
25	5.27955457130247e-12\\
26	4.01489952395195e-12\\
27	4.69813077330627e-12\\
28	-4.59909887950971e-12\\
29	4.63773464076667e-12\\
30	4.4406700538957e-12\\
};
\addlegendentry{$\beta^\text{half-lin}$}
\end{axis}

\end{tikzpicture}%

%% file: comp_132.tex
\definecolor{mycolor1}{rgb}{0.00000,0.44700,0.74100}%
\definecolor{mycolor2}{rgb}{0.85000,0.32500,0.09800}%
\definecolor{mycolor3}{rgb}{0.92900,0.69400,0.12500}%
%\definecolor{mycolor4}{rgb}{0.49400,0.18400,0.55600}%
\definecolor{mycolor4}{rgb}{0.1,0.7,0.15}%
\begin{tikzpicture}

\begin{axis}[%
width=3.4in,
height=1.4in,
at={(0.5in,0.4in)},
scale only axis,
line width=1,
xmin=0.05,
xmax=40,
xlabel style={font=\color{white!15!black}},
xlabel={prediction horizon $N$},
ymode=log,
ymin=9e-15,
ymax=1,
ytick={0.00000000000001, 0.0000000001, 0.00001,1},
yminorticks=true,
ylabel style={font=\color{white!15!black}},
ylabel={ $J_\infty^\mathrm{av}(x_0,\mu_N^\beta)-\ell^\star$},
axis background/.style={fill=white},
xmajorgrids,
ymajorgrids,
yminorgrids,
legend columns = 2,
legend style={xshift = 0.18cm, yshift=-2.37cm, legend cell align=left, align=left, draw=white!15!black}%, anchor=south east}
]
% N = 25

\addplot [only marks,mark=*,line width=1, mark options={solid, mycolor3, line width=1pt}]
  table[row sep=crcr]{%
  	1	0.685458831821469\\
  	2	0.172532129555068\\
  	3	0.068247055826915\\
  	4	0.0350720453964413\\
  	5	0.0209843310438171\\
  	6	0.0138690926168494\\
  	7	0.0098228711263979\\
  	8	0.0073151966542242\\
  	9	0.00565730818671639\\
  	10	0.00450519745037914\\
  	11	0.00367239432813937\\
  	12	0.00305095616627948\\
  	13	0.00257495449778178\\
  	14	0.00220229448082776\\
  	15	0.00190508028831271\\
  	16	0.0016642330756953\\
  	17	0.00146634623776665\\
  	18	0.00130177543878474\\
  	19	0.00116343899056992\\
  	20	0.00104604194797941\\
  	21	0.000945561114852556\\
  	22	0.000858895291368178\\
  	23	0.000783622825109687\\
  	24	0.000717830421509191\\
  	25	0.000659990241490771\\
  	26	0.00060887032282908\\
  	27	0.000563468383898424\\
  	28	0.000522962285063322\\
  	29	0.000486672523295928\\
  	30	0.000454033531472886\\
  	31	0.000424571496715931\\
  	32	0.000397887058666502\\
  	33	0.000373641698113492\\
  	34	0.000351546942985781\\
  	35	0.000331355744372708\\
  	36	0.000312855537905943\\
  	37	0.000295862624295484\\
  	38	0.000280217589978538\\
  	39	0.000265781553529632\\
  	40	0.0002524330719198\\
};
\addlegendentry{$\beta^\mathrm{lin}$}

\addplot [only marks, mark size=3pt, mark=x, mark options={solid, mycolor4, line width=1pt}]
  table[row sep=crcr]{%
  	1	0.685458831821469\\
  	2	0.0727583165970183\\
  	3	0.0185658442326662\\
  	4	0.00645766953533888\\
  	5	0.00271671063583678\\
  	6	0.00131163976516624\\
  	7	0.000703189293964623\\
  	8	0.000408889346349417\\
  	9	0.000253449493740288\\
  	10	0.000165327033670781\\
  	11	0.000112406369038087\\
  	12	7.90829790433545e-05\\
  	13	5.72543703789563e-05\\
  	14	4.24708999480217e-05\\
  	15	3.21697433780521e-05\\
  	16	2.48133342699575e-05\\
  	17	1.94461157754233e-05\\
  	18	1.54557186946569e-05\\
  	19	1.24389808364711e-05\\
  	20	1.01240658030299e-05\\
  	21	8.32374688286563e-06\\
  	22	6.90659039004338e-06\\
  	23	5.77872959839354e-06\\
  	24	4.87207425559077e-06\\
  	25	4.13652609565496e-06\\
  	26	3.53474298675494e-06\\
  	27	3.03855711569767e-06\\
  	28	2.62648629223428e-06\\
  	29	2.28197985929768e-06\\
  	30	1.99216591356688e-06\\
  	31	1.7469455468877e-06\\
  	32	1.53833048832652e-06\\
  	33	1.3599535788611e-06\\
  	34	1.20670337455664e-06\\
  	35	1.07444886987551e-06\\
  	36	9.59830301461295e-07\\
  	37	8.60098862354164e-07\\
  	38	7.72992943209871e-07\\
  	39	6.96641882402815e-07\\
  	40	6.29490591652626e-07\\
};
\addlegendentry{$\beta^\text{poly}$}

\addplot [color=mycolor1,mark=*,solid,line width=1]
  table[row sep=crcr]{%
  	1	0.685458831821469\\
  	2	0.0256512326473008\\
  	3	0.00235508659599493\\
  	4	0.00025395832886721\\
  	5	2.86956636026403e-05\\
  	6	3.29189894876514e-06\\
  	7	3.79557375529416e-07\\
  	8	4.38381886347372e-08\\
  	9	5.06617592321845e-09\\
  	10	5.85589354784588e-10\\
  	11	6.76920741682352e-11\\
  	12	7.82529596676795e-12\\
  	13	9.04609720464578e-13\\
  	14	1.04805053524615e-13\\
  	15	1.24344978758018e-14\\
  	16	1.55431223447522e-15\\
  	17	4.44089209850063e-16\\
  	18	4.44089209850063e-16\\
  	19	4.44089209850063e-16\\
  	20	4.44089209850063e-16\\
  	21	0\\
  	22	4.44089209850063e-16\\
  	23	4.44089209850063e-16\\
  	24	4.44089209850063e-16\\
  	25	4.44089209850063e-16\\
  	26	4.44089209850063e-16\\
  	27	4.44089209850063e-16\\
  	28	4.44089209850063e-16\\
  	29	0\\
  	30	4.44089209850063e-16\\
  	31	4.44089209850063e-16\\
  	32	0\\
  	33	0\\
  	34	0\\
  	35	0\\
  	36	0\\
  	37	0\\
  	38	0\\
  	39	0\\
  	40	0\\
};
\addlegendentry{$\beta^\text{un}\ \ $}

\addplot [only marks,mark=o,dashed,line width=1, mark options={solid, mycolor2, line width=1pt}]
table[row sep=crcr]{%
	1	0.685458831821469\\
	2	0.0256512326473008\\
	3	0.00715757888908009\\
	4	0.00107164139760108\\
	5	0.00034195634612022\\
	6	6.24666681365849e-05\\
	7	2.13611568928851e-05\\
	8	4.24956240774321e-06\\
	9	1.52641835216016e-06\\
	10	3.19488373978771e-07\\
	11	1.19049116120706e-07\\
	12	2.57929682057068e-08\\
	13	9.88066384266517e-09\\
	14	2.19492202191418e-09\\
	15	8.58754845012299e-10\\
	16	1.94396498898186e-10\\
	17	7.73137109888467e-11\\
	18	1.77597936357188e-11\\
	19	7.15583148291898e-12\\
	20	1.66311409088848e-12\\
	21	6.77458089626271e-13\\
	22	1.58761892521397e-13\\
	23	6.52811138479592e-14\\
	24	1.55431223447522e-14\\
	25	6.21724893790088e-15\\
	26	1.55431223447522e-15\\
	27	6.66133814775094e-16\\
	28	4.44089209850063e-16\\
	29	4.44089209850063e-16\\
	30	4.44089209850063e-16\\
	31	0\\
	32	0\\
	33	4.44089209850063e-16\\
	34	0\\
	35	4.44089209850063e-16\\
	36	4.44089209850063e-16\\
	37	4.44089209850063e-16\\
	38	4.44089209850063e-16\\
	39	0\\
	40	0\\
};
\addlegendentry{$\beta^\text{half-lin}$}
\end{axis}

\end{tikzpicture}%

%% file: sVII_concl_gd.tex
\section{Conclusion and outlook}
\label{s:cao}

We have presented in this paper a discounted E-MPC scheme without terminal conditions for nonlinear systems which behave optimally on a periodic orbit.
We provided requirements on discount functions such that practical asymptotic stability and asymptotic average performance can be guaranteed.
As seen in Example~\ref{ex:ecgr}, the quantitative performance level can vary strongly for different discounts and the performance of the linear discount from \cite{1-schwenkel23} can be outperformed significantly.
Based on the findings in our numerical examples we conjecture that $\beta^\mathrm{half-lin}$ is generally a good choice when dealing with periodic orbits, i.e., linearly discounting only the second half of the prediction horizon.
When dealing with steady states (optimal period length $p^\star =1$), we found in our examples no advantage of using discounts instead of the standard undiscounted E-MPC scheme.
However, a more extensive numerical study is required to draw conclusions about which discount is best in which situation.
Additionally, it is interesting to theoretically investigate under which assumption on the discount function an exponential convergence rate is guaranteed, as is the case for optimal steady-state operation without discounting~\cite{5-gruene14}.

%% file: sVIII_appx_gd.tex
\appendix
\section{Proof of Lemmas~\ref{lem:almost_optimal_candidate} and~\ref{lem:diff_value}}
\label{app:lemmas}
In order to prove Lemmas~\ref{lem:almost_optimal_candidate} and~\ref{lem:diff_value}, we first show in Lemma~\ref{thm:wtp} that the weak turnpike property (cf.~\cite[Definition~7]{1-schwenkel23}) holds and we use it in Lemma~\ref{lem:average} to compute the average of a function over turnpike trajectories.

\begin{lemma}[Weak turnpike property]
	\label{thm:wtp}
	Let Assumptions~\ref{ass:orbit}-\ref{ass:df} hold.
	We define the number of points of the optimal trajectory $\uNx$ in an $\varepsilon$-neighborhood of the optimal orbit $\Pi^\star$ as $
		Q_\varepsilon^\beta(N,\,x) := \#\left\lbrace k \in \mathbb{I}_{[0,N-1]} \big| \big\lVert \big(\xNx(k),\uNx(k)\big)\big\rVert_{\Pi^\star} \leq \varepsilon \right\rbrace$.
	Then, there exists $\alpha \in \mathcal{K}_\infty$ such that for all $x\in\mathbb{X}$, $N \in\N$, and $\varepsilon > 0$ we have
	\begin{equation}
		Q_\varepsilon^\beta(N,\,x) \geq N -\textstyle\frac{\sqrt{N}}{\alpha(\varepsilon)}.
		\label{eq:wtp-Q}
	\end{equation}
\end{lemma}
\begin{proof}
	First, let us establish an upper bound on the rotated cost functional $\tilde J_N^\beta (x,\uNx)=\sum_{k=0}^{N-1} \beta \left(\frac kN \right)\tilde \ell\big(\xNx(k),\uNx(k)\big)$ starting with
	\begin{align*}
		\tilde{J}_N^\beta(x,\uNx) &\overset{\eqref{eq:s-diss}}{=} V_N^\beta(x) - \sum_{k=0}^{N-1} \betakN{k} \ell^\star + \lambda(x)- \sum_{k=1}^N \left(\betakN{k-1}-\betakN{k}\right)\lambda(\xNx(k)).
	\end{align*}
	Due to Assumptions~\ref{ass:s-diss},~\ref{ass:cac}, and~\ref{ass:df} we have
	\begin{align*}
		-\sum_{k=1}^N \left(\betakN{k-1}-\betakN{k}\right)\lambda(\xNx(k)) &\leq \sum_{k=1}^N \left(\betakN{k-1}-\betakN{k}\right) \bar{\lambda}=\underbrace{\left(\beta(0)-\beta(1)\right)}_{=1}\bar \lambda.
	\end{align*}
	Due to Assumptions~\ref{ass:lctrb} and~\ref{ass:ftr} we can reach the optimal orbit $\PiXstar$ from $x$ in at most $M_3=M_1+M_2$ steps.
	Therefore, Lemma~22 from~\cite{1-schwenkel23} yields
	\begin{align*}
		V_N^\beta(x) - \sum_{k=0}^{N-1} \betakN{k} \ell^\star + \lambda(x)+\bar \lambda \leq M_3(\bar \ell-\ell^\star)+2\bar \lambda+\ell^\star p^\star=:C.
	\end{align*}
	Although this Lemma was only stated for linear discounts, it is straightforward to see that the proof is valid for general $\beta$ satisfying Assumption~\ref{ass:df} as only the fact that $\beta$ is non-increasing and $\beta(1)=0$, $\beta(0)=1$ is used.
	Furthermore, due to Assumption~\ref{ass:df} we know that $\mathrm{lim\,sup}_{\xi\to 1} \beta'(\xi)<0$ and $\beta$ is non-incresing, hence there must exist a constant $c>0$ such that $\beta$ is lower bounded by the scaled linear discount $\beta(\xi)\geq c \beta^{\text{lin}}(\xi)$ for all $\xi\in[0,1]$.
	Hence, as $\tilde \ell$ is non-negative, we obtain $\tilde J_{N}^{\beta^\text{lin}}(x,\uNx) \leq c^{-1} \tilde J_{N}^{\beta}(x,\uNx) \leq c^{-1} C$.
	Therefore, we can apply Lemma~23 from~\cite{1-schwenkel23}, which guarantees that the weak turnpike property holds with $\alpha(\varepsilon)= \sqrt{\frac{c}{2C} \underline\alpha_{\tilde \ell} (\varepsilon)}$, $\alpha\in\Kinf$. \hfill \qed
\end{proof}

\begin{lemma}[Average over turnpike trajectories]\label{lem:average}
	Let Assumptions~\ref{ass:orbit}--\ref{ass:df} hold and let $g:\X\times \U\to \R$ be a continuous function and let $g^\star :=\frac 1 {p^\star} \sum_{k=0}^{p^\star -1} g(\Pistar(k))$, then there exists $\delta_1\in\L$ such that for all $x\in\X$ 
	\begin{align}\label{eq:average}
		\left| g^\star - \sum_{k=1}^{N-1} \left(\betakN{k-1}-\betakN{k}\right) g\left(\xNx(k), \uNx(k)\right) \right| \leq \delta_1(N).
	\end{align}
\end{lemma}
\begin{proof}
	Due to Lemma~\ref{thm:wtp} the weak turnpike property~\eqref{eq:wtp-Q} holds, which yields with $\varepsilon=\sigma(N):= \alpha^{-1}(\sqrt[-4]{N})$, $\sigma \in \L$ that for all $N\in \N$ and $x\in\X$ we have
	\begin{equation}
		Q_{\sigma(N)}^\beta(N,x) \geq N-\sqrt[4]{N^3}.
		\label{eq:QsN}
	\end{equation}
	Hence, there are at most $\sqrt[4]{N^3}$ points of $(\xNx, \uNx)$ outside the $\sigma(N)$-neighborhood of $\Pistar$.
	Denote the left hand side of~\eqref{eq:average} by $\Sigma_1$ and let us split the sum in $\Sigma_1$ over $k\in \I_{[1,N-1]}$ into pieces $k\in\mathcal I_j=\I_{[j p^\star+1,(j+1)p^\star]}$ of length $p^\star$ for $j \in\I_{[0,\lfloor (N-1)/p^\star \rfloor-1]}$ and let $\mathcal I_{\lfloor (N-1)/p^\star \rfloor}=\I_{[\lfloor (N-1)/p^\star \rfloor p^\star, N-1]}$ contain the remainder. 
	Let $\mathcal J$ be the set of indices $j$ for which the piece $\mathcal I_j$ is not fully inside the $\sigma(N)$-neighborhood of $\Pistar$ and of the remainder piece, hence $\#\mathcal J \leq \sqrt[4]{N^3}+1$.
	Now, using Lipschitz continuity $\left|\betakN{k-1}-\betakN{k}\right|\leq \frac{L_\beta}{N}$, we bound all pieces in $\mathcal J$ by
	\begin{align*}
		\bigg|\sum_{j\in\mathcal J} \sum_{k\in \mathcal I_j} \left(\betakN{k-1}-\betakN{k}\right) g\left(\xNx(k), \uNx(k)\right) \bigg| \leq \left( \sqrt[4]{N^3}+1\right) p^\star \bar g \frac{L_\beta}{N} =:\delta_4(N)
	\end{align*}
	where continuity of $g$ and compactness of $\X \times \U$ imply $\bar g = \max_{(x,u)\in\X\times \U} |g(x,u)|$ exists and is finite. 
	Note that $\delta_4\in\L$.
	Further, with the same reasons there exists $\overline \alpha_g\in\Kinf$ that satisfies $|g(x,u) - g(\Pistar(k))|\leq \overline\alpha_g(\|(x,u)-\Pistar(k)\|)$ for all $k\in\I_{[0,p^\star-1]}$ and all $(x,u)\in\X\times \U$.
	Next, define the complement of $\mathcal J$ by $\bar {\mathcal J} := \I_{[0,\lfloor (N-1)/p^\star \rfloor]} \setminus \mathcal J$, and compute
	\begin{align*}
		\Sigma_1 & \leq \bigg|g^\star - \sum_{j\in\bar{\mathcal J}} \sum_{k\in \mathcal I_j} \left(\betakN{k-1}-\betakN{k}\right) g\big(\xNx(k), \uNx(k)\big) \bigg|+\delta_4(N)\\
		&\leq 
		\bigg|g^\star - \sum_{j\in\bar{\mathcal J}} \sum_{k\in \mathcal I_j} \left(\betakN{k-1}-\betakN{k}\right) g\big(\Pistar(i_k)\big) \bigg| + \underbrace{L_\beta \overline\alpha_g\big(\sigma(N)\big)+\delta_4(N)}_{=:\delta_5(N)}
	\end{align*}
	where $i_k\in\I_{[0,p^\star -1]}$ is such that $\|\Pistar(i_k)-\big(\xNx(k),\uNx(k)\big)\|\leq \sigma(N)$ and $\delta_5\in\L$.
	First, we consider $N\geq N_0$ and choose $N_0\in\N$ large enough such that $\sigma(N_0)\leq \bar \varepsilon$ with $\bar \varepsilon$ from~\cite[Lemma~15]{3-mueller16}.
	Then this Lemma guarantees that the points in $\mathcal I_j$ follow the orbit, since they are inside the $\sigma(N)$-neighborhood of $\Pistar$, i.e., $i_{k+1}=[i_k +1]_{p^\star}$ for all $k\in\mathcal I_j$. 
	Hence, the $p^\star$ elements of $\{i_k|k\in\mathcal I_j\}$ are exactly $[0,p^\star-1]$ and for each $p\in[0,p^\star-1]$ and $j\in\bar{\mathcal J}$ there exists a $k_{p,j}\in\mathcal I_j$ such that $i_{k_{p,j}}=p$.
	Therefore we can continue as follows 
	\begin{align}\nonumber
		\Sigma_1 &\leq  \bigg|\sum_{p=0}^{p^\star-1} g\big(\Pistar(p)\big) \bigg(\frac 1{p^\star} - \sum_{j\in\bar{\mathcal J}} \left(\betakN{k_{p,j}-1}-\betakN{k_{p,j}}\right) \bigg) \bigg| + \delta_5(N)\\
		&\leq  \sum_{p=0}^{p^\star-1} \left|g\big(\Pistar(p)\big)\right| \bigg|\frac 1{p^\star} - \sum_{j\in\bar{\mathcal J}} \left(\betakN{k_{p,j}-1}-\betakN{k_{p,j}}\right) \bigg|  + \delta_5(N). \label{eq:average_intermediate}
	\end{align}
	Since $\beta$ is piecewise continuously differentiable, we can use the mean value theorem (on all differentiable pieces separately and taking the infimum) to obtain
	\begin{align*}
		 \sum_{j\in\bar{\mathcal J}} \left(\betakN{k_{p,j}-1}-\betakN{k_{p,j}}\right) \leq \sum_{j\in\bar{\mathcal J}}  -\frac 1 N \inf_{\xi \in \left[\frac{k_{p,j}-1}{N},\frac{k_{p,j}}{N}\right]} \beta'(\xi)
	\end{align*}
	where in the $\inf$ we ignore the finitely many values $\xi$ in the interval at which $\beta'$ does not exist.
	Further, since $ \left[\frac{k_{p,j}-1}{N},\frac{k_{p,j}}{N}\right]\subseteq \left[\xi_j,\xi_{j+1}\right]$ with $\xi_j:=\frac{j p^\star}{N}$ for $j\in\I_{[0,\lfloor (N-1)/p^\star \rfloor]}$ and $\xi_{\lfloor (N-1)/p^\star \rfloor+1}:= 1$, we can continue
	\begin{align*}
		\sum_{j\in\bar{\mathcal J}}  -\frac 1 N \inf_{\xi \in \left[\frac{k_{p,j}-1}{N},\frac{k_{p,j}}{N}\right]} \beta'(\xi) \leq -\sum_{j=0}^{\lfloor (N-1)/p^\star \rfloor} \frac{\xi_{j+1}-\xi_j}{p^\star} \inf_{\xi \in \left[\xi_j,\xi_{j+1}\right]} \beta'(\xi)
	\end{align*}
	where we used that $\beta'(\xi) \leq 0$ as $\beta$ is non-increasing.
	The partition of the interval $[0,1]$ defined by $\xi_j$ satisfies $\xi_{j+1}-\xi_j\leq \frac {p^\star} N \to 0$ as $N \to \infty$ and thus the sum converges to the Riemann integral of $\beta'$ as $N\to \infty$, i.e., there exists $\delta_6\in\L$ such that
	\begin{align*} 
		-\sum_{j=0}^{\lfloor (N-1)/p^\star \rfloor} \frac{\xi_{j+1}-\xi_j}{p^\star} \inf_{\xi \in \left[\xi_j,\xi_{j+1}\right]} \beta'(\xi) &\leq -\frac 1 {p^\star} \int_0^1 \beta'(\xi)\,\mathrm{d} \xi + \delta_6(N) =\frac{1}{p^\star}+\delta_6(N)
	\end{align*}
	where we used $-\int_0^1 \beta'(\xi)\,\mathrm{d} \xi=\beta(0)-\beta(1)=1$. With similar arguments we can construct a lower bound and obtain
	\begin{align*}
		\frac{1}{p^\star}-\delta_6(N) \leq \sum_{j\in\bar{\mathcal J}} \left(\betakN{k_{p,j}-1}-\betakN{k_{p,j}}\right) \leq \frac{1}{p^\star}+\delta_6(N).
	\end{align*}
	Using this in~\eqref{eq:average_intermediate} and defining $\delta_1(N) := \delta_5(N)+\sum_{p=0}^{p^\star-1} \left|g(\Pistar(p))\right| (\delta_6(N))$, which satisfies $\delta_1\in\L$, we obtain the desired result for $N\geq N_0$.
	For $N< N_0$ we know that $\Sigma_1\leq |g^\star| + \bar g$ and thus by redefining $\delta_1 (N) := \max\{\delta_1(N_0), |g^\star| + \bar g\}$ for $N<N_0$ the result follows for all $N\in\N$ as still $\delta_1\in\L$.   \hfill \qed
	\end{proof}
	\begin{proof}[Proof of Lemma~\ref{lem:almost_optimal_candidate}]
	To begin with, we construct a candidate solution for $V_{N}^\beta\left(\xNx(1)\right)$ based on $\uNx$.
	Therefore, we use the weak turnpike property and $\sigma(N) = \alpha^{-1}(\sqrt[-4]{N})$ to obtain~\eqref{eq:QsN}.
	Let $P\in \I_{[0,N-1]}$ be the largest time point where $(\xNx,\uNx)$ is in the $\sigma(N)$-neighborhood of $\Pistar$, then \eqref{eq:QsN} implies $P\geq N-\sqrt[4]{N^3}-1$.
	First, we consider only $N\geq N_0$ and choose $N_0\in\N $ large enough such that $P\geq 1$ and $\sigma(N)\leq \kappa$ for all $N\geq N_0$ with $\kappa$ from Assumption~\ref{ass:lctrb}.
	Hence, at $P$ the point $\xNx(P)$ is in the local controllability neighborhood of $\PiXstar$ such that we know that there exists $u_1\in \U^{M_1}(\xNx(P))$ with $x_{u_1}(M_1,\xNx(P))=\xNx(P)$.
	Thus, for $V_{N}^\beta(\xNx(1))$ we can construct a candidate solution $\bar u \in \U^{N}(\xNx(1))$ and the resulting state trajectory $x_{\bar u}$  based on the solution for $V_{N}^\beta(x)$ as follows
	\begin{align}
		\bar{u}(k) &= \begin{cases}
			\uNx(k+1) &\text{for } k \in \mathbb{I}_{[0,\,P-2]} \\
			u_1(k-P+1) &\text{for } k \in \mathbb{I}_{[P-1,\,P+M_1-2]} \\
			\uNx(k-M_1+1) &\text{for } k \in \mathbb{I}_{[P+M_1-1,\,N-1]}.
		\end{cases}
		\label{eq:cs-u}\\
		x_{\bar u}(k,\xNx(1)) &= \begin{cases}
			\xNx(k+1) &\text{for } k \in \mathbb{I}_{[0,P-2]} \\
			x_{u_1}\big(k-P+1,\xNx(P)\big) &\text{for } k \in \mathbb{I}_{[P-1,P+M_1-2]} \\
			\xNx(k-M_1+1) &\text{for } k \in \mathbb{I}_{[P+M_1-1,N]}.
		\end{cases}
		\label{eq:cs-x}
	\end{align}
	With this candidate solution we can upper bound $V_{N}^\beta \big(\xNx(1)\big) \leq J_{N}^\beta \big(\xNx(1),\bar u\big)=\Sigma_2+\Sigma_3+\Sigma_4$ where $\Sigma_2$, $\Sigma_3$, $\Sigma_4$ sum up the three pieces in~\eqref{eq:cs-u} and~\eqref{eq:cs-x}
	\begin{align*}
		\Sigma_2 &:= \sum_{k=0}^{P-2} \betakN{k} \ell\big(\xNx(k+1),  \uNx(k+1)\big)\\
		\Sigma_3 &:= \sum_{k=P-1}^{P+M_1-2} \betakN{k} \ell\big(x_{u_1}\left(k-P+1,\xNx(P)\right), u_1(k-P+1)\big)\\
		\Sigma_4 &:= \sum_{k=P-1}^{N-M_1-1} \betakN{k+M_1} \ell\big(\xNx(k+1),  \uNx(k+1)\big).
	\end{align*}
	We start to bound $\Sigma_3$ as follows
	\begin{align*}
	\Sigma_3 \leq M_1 \bar\ell \betakN{P-1}\leq M_1 \bar\ell\betakN{N-\sqrt[4]{N^3}-2} = M_1 \bar\ell \beta \left({\textstyle 1 - \frac{1}{\sqrt[4]{N}}-\frac 2N }\right)=:\delta_7(N)
	\end{align*}
	where $\delta_7\in\L$ due to $\beta(1)=0$ and $\beta$ being continuous.
	Further, as $\beta$ is non-increasing and $\beta$, $\ell$ are non-negative, we can bound $\Sigma_2+\Sigma_4$ as follows
	\begin{align*}
		\Sigma_2+\Sigma_4 &\leq \sum_{k=1}^{N-1} \betakN{k-1} \ell\left(\xNx(k), \uNx(k)\right) \\
		& =V_N^\beta(x)-\ell(x,\uNx(0)) + \sum_{k=1}^{N-1} \left(\betakN{k-1}-\betakN{k}\right) \ell\left(\xNx(k), \uNx(k)\right)\\
		&\refeq{\text{Lemma~\ref{lem:average}}}{\leq}\quad V_N^\beta(x)-\ell(x,\uNx(0)) + \ell^\star + \delta_1(N)
	\end{align*}
	for some $\delta_1\in\L$.
	Define $\delta_2(N):=\delta_1(N)+\delta_7(N)$ then the desired inequality~\eqref{eq:almost_optimal_candidate} follows for $N\geq N_0$.
	For $N< N_0$, we know that~\eqref{eq:almost_optimal_candidate} holds with $\delta_2 (N) := \max\{(N_0+1)\bar\ell,\delta_2(N_0)\}$ and since this guarantees $\delta_2\in\L$, the desired result follows for all $N\in\N$.  \hfill \qed
	\end{proof}

	\begin{proof}[Proof of Lemma~\ref{lem:diff_value}]
		This Lemma is the generalization of \cite[Lemma~29]{1-schwenkel23} from linear discounts to discounts that satisfy Assumption~\ref{ass:df}. 
		Also the proof follows the same line of arguments.
		First, note that if we replace the stage cost $\ell$ with the rotated stage cost $\tilde \ell$ from~\eqref{eq:s-diss}, then the Assumptions~\ref{ass:s-diss}--\ref{ass:df} still hold and thus Lemma~\ref{thm:wtp} applies and solutions of $\tilde V_{N}^\beta $ satisfy the weak turnpike property as well.
		The weak turnpike property \eqref{eq:wtp-Q} for $V_N^\beta (x)$, $V_N^\beta (y)$, $\tilde V_N^\beta (x)$, and $\tilde V_N^\beta (y)$ guarantees that there are at least $N-\frac{4\sqrt{N}}{\alpha(\varepsilon)}$ time points $k\in \I_{[0,N-1]}$ where these four trajectories are $\varepsilon$ close to $\Pistar$.
		Choosing $\varepsilon=\sigma_2(N):=\alpha^{-1} (4\sqrt[-4]{N})$ yields that there are at least $N-\sqrt[4]{N^3}$ such points.
		We denote the largest one by $P\geq N-\sqrt[4]{N^3}-1$ and choose $N_0$ large enough such that $P\geq 1$ and $\sigma_2(N)\leq \kappa$ for $N\geq N_0$.
		First, consider only $N\geq N_0$, then we use Assumption~\ref{ass:lctrb} to construct a candidate solution $\bar u\in\U^{N}(y)$ for $\tilde V_N^\beta (y)$ which consists of the first $P$ steps of the optimal solution of $V_N^\beta(y)$ and goes afterwards in $M_1$ steps to $\Pistar$ using the local controllability and then remains there. 
		Using this candidate solution, we can upper bound
		\begin{align}\label{eq:sim_cost1}
			\tilde V_N^\beta (y) \leq \tilde J_{N}^{\beta} (y, \bar  u) \leq \sum_{k=0}^{P-1} \betakN{k} \tilde \ell\left(\xNx[y](k),\uNx[y](k)\right) + \delta_8(N)
		\end{align}
		where we used $\tilde \ell ( \Pi^\star (k))=0$ (see~\cite[Lemma~27]{1-schwenkel23}) and~\eqref{eq:A-lC} from Assumption~\ref{ass:lctrb} to compute $\delta_8(N):=M_1\overline\alpha_{\tilde \ell}(\alpha_c (\sigma_2(N)))$, which is $\delta_8\in\mathcal L$.
		Furthermore, non-negativity of $\tilde \ell$ implies
		\begin{align}\label{eq:sim_cost2}
			\tilde V_N^\beta (x) \geq \sum_{k=0}^{P-1} \betakN{k} \tilde \ell\left(\xNxrot(k),\uNxrot(k)\right).
		\end{align}
		Therefore, using the definition of the rotated cost~\eqref{eq:s-diss}, we have
		\begin{align}\nonumber
			\tilde V_N^\beta (y)-\tilde V_N^\beta (x) &\leq \delta_8(N)+ \underbrace{\sum_{k=0}^{P-1} \betakN{k} \left(\ell\left(\xNx[y](k),\uNx[y](k)\right)-  \ell\left(\xNxrot(k),\uNxrot(k)\right) \right)}_{=:\Sigma_5}  \\\nonumber
			&\quad +\lambda(y)-\lambda(x)+ \underbrace{\sum_{k=1}^{P-1} \left(\betakN{k-1}-\betakN{k}\right)\left(\lambda(\xNxrot(k))- \lambda(\xNx[y](k))\right)}_{=:\Sigma_6}\\\label{eq:sim_cost3}
			&\quad +\betakN{P-1}\left(\lambda(\xNxrot(P))-\lambda(\xNx[y](P))\right).
		\end{align}
		Let us investigate these terms further starting with
		\begin{align*}
			\betakN{P-1}\left(\lambda(\xNxrot(P))-\lambda(\xNx[y](P))\right)\leq 2\bar\lambda \betakN{P-1} \leq 2\bar\lambda  \beta \left({\textstyle 1 - \frac{1}{\sqrt[4]{N}}-\frac 2N }\right)
		\end{align*}
		and define $\delta_{9}(N)=2\bar\lambda  \beta \left({\textstyle 1 - \frac{1}{\sqrt[4]{N}}-\frac 2N }\right)$ and note that $\delta_{9}\in\L$ as $\beta(1)=0$ and $\beta$ is continuous.
		Now let us turn to $\Sigma_6$ and use Lemma~\ref{lem:average} to obtain that there exists $\delta_{10}\in \L$ with $\Sigma_6\leq \delta_{10}(N)$, where the assumptions are satisfied since $\lambda$ is continuous and the weak turnpike property holds not only for $\xNx[y],\uNx[y]$ but also for $\xNxrot,\uNxrot$. 
		The fact that the sum in $\Sigma_6$ goes only until $P-1$ and not until $N-1$ just implies that we get another $\frac{L_\beta}{N} (\sqrt[4]{N^3}+1) 2 \bar \lambda\in\mathcal L$ on top of $\delta_{10}$, as $N-P\leq \sqrt[4]{N^3}+1$.
		Next, we consider the part $\Sigma_5$ and construct a candidate solution $\bar u\in\U^{N}(x)$ for $V_N^\beta (x)$ which consists of the first $P$ steps of the optimal solution $\uNxrot$ of $\tilde V_N^\beta(x)$ and afterwards uses the local controllability to go to $x_{\uNx[y]}(P,y)$ in $K\in\I_{[M_1, M_1+p^\star -1]}$ steps\footnote{Since $\xNxrot(P)$ and $\xNx[y](P)$ are both in the $\sigma_2(N)$ neighborhood of $\PiXstar$, but not necessarily in the $\sigma(N)$-neighborhood of the same $\PiXstar(j)$, we need at most $p^\star-1$ additional steps to have the same phase.} and then applies $\uNx[y](k)$ starting at $k=P$ to remain on $V_N^\beta (y)$. 
		Hence,
		\begin{align*}
			\Sigma_5 &= \sum_{k=0}^{P-1} \betakN{k} \ell\left(\xNx[y](k),\uNx[y](k)\right)  - J_N^\beta (x,\bar u) + \sum_{k=P}^{P+K-1} \betakN{k} \ell\left( x_{\bar u}(k,x),\bar u(k)\right) \\&\qquad + \sum_{k=P}^{N-K-1} \betakN{k+K} \ell\left(\xNx[y](k),\uNx[y](k)\right) \\
			&\leq V_{N}^\beta (y)-V_N^\beta(x) + \betakN{P}K\bar \ell,
		\end{align*}
		where the last inequality holds as $\ell$ and $\beta$ are non-negative and $\beta$ is non-increasing.
		As $P\geq N-\sqrt[4]{N^3}-1$, we know that $\betakN{P}K\bar \ell \leq \betakN{N-\sqrt[4]{N^3}-1}K\bar \ell=:\delta_{12}(N)$ where $\delta_{12}\in\L$ as $\beta$ is continuous and $\beta(1)=0$.
		Finally, define $\delta_3(N) := \delta_9(N)+\delta_{9}(N)+\delta_{10}(N)+\delta_{12}(N)$ for $N\geq N_0$ and plug the bounds on $\Sigma_5$ and $\Sigma_6$ into~\eqref{eq:sim_cost3} to obtain the desired result for $N\geq N_0$. As~\eqref{eq:Lemma29} holds with $\delta_3(N):=\max\{\delta_3(N_0), N_0(2\bar \ell + 4\bar \lambda)+2\bar\lambda\}$ for $N<N_0$, which guarantees $\delta_3\in\L$, the desired result follows for all $N\in\N$.   \hfill \qed
	\end{proof}